\documentclass[hidelinks,11pt]{amsart}

\usepackage{amsmath,amssymb,amsfonts,hyperref,graphicx,tikz,mathrsfs,amsaddr,blindtext}

\usepackage[labelsep=period, font=footnotesize, labelfont=bf]{caption}

\makeatletter
\@namedef{subjclassname@2020}{2020 Mathematics Subject Classification}
\makeatother

\usepackage[margin=1in,footskip=0.25in]{geometry}

\theoremstyle{plain}

\newtheorem{theorem}{Theorem}[section]
\newtheorem{lemma}[theorem]{Lemma}

\theoremstyle{definition}
\theoremstyle{proposition}
\newtheorem{definition}[theorem]{Definition}

\newtheorem{corollary}[theorem]{Corollary}

\newtheorem{remark}[theorem]{Remark}

\numberwithin{equation}{section}

\def\C{\mathbb{C}}

\def\A{\mathcal{A}}

\def\I{\mathcal{I}}
\def\L{\mathcal{L}}
\def\Q{\mathcal{Q}}

\def\v{\mathbf{v}}

\def\w{\mathbf{w}}

\def\x{\mathbf{x}}

\def\y{\mathbf{y}}

\def\s{\mathfrak{s}}

\def \la{\lambda}

\date{\today}

\begin{document}
\title[Spectra of weighted uniform hypertrees]
{Spectra of weighted uniform hypertrees}


\author{Jiang-Chao Wan, Yi Wang, Fu-Tao Hu}
\address{\rm{\normalsize{Center for Pure Mathematices, School of Mathematical Sciences, Anhui University, \\ Hefei 230601, Anhui, China}}}
\thanks{{\it E-mail address:} wanjc@stu.ahu.edu.cn (J.-C. Wan), wangy@ahu.edu.cn (Y. Wang), hufu@ahu.edu.cn(F.-T. Hu)}
\thanks{{\it Corresponding author.} Yi Wang}
\thanks{{\it Funding.} Supported by the National Natural Science Foundation of China (No. 12171002, 11871073) and Anhui Provincial Natural Science Foundation (No. 2108085MA02)}

\maketitle

\begin{abstract}
Let $T$ be a $k$-tree equipped with a weighting function $\w: V(T)\cup E(T)\rightarrow \C$, where $k \geq 3$.
The weighted matching polynomial of the weighted $k$-tree $(T,\w)$ is defined to be
$$
\mu(T,\w,x)= \sum_{M \in \mathcal{M}(T)}(-1)^{|M|}\prod_{e \in E(M)}\mathbf{w}(e)^k \prod_{v \in V(T)\backslash V(M)}(x-\w(v)),
$$
where $\mathcal{M}(T)$ denotes the set of matchings (including empty set) of $T$.
In this paper, we investigate the eigenvalues of the adjacency tensor $\A(T,\w)$ of the weighted $k$-tree $(T,\w)$.
The main result provides that $\w(v)$ is an eigenvalue of $\A(T,\w)$ for every $v\in V(T)$,
and if $\lambda\neq \w(v)$ for every $v\in V(T)$, then $\lambda$ is an eigenvalue of $\A(T,\w)$ if and only if
there exists a subtree $T'$ of $T$ such that $\lambda$ is a root of  $\mu(T',\w,x)$.
Moreover, the spectral radius of $\A(T,\w)$ is equal to the largest root of $\mu(T,\w,x)$ when $\w$ is real and nonnegative.
The result extends a work by  Clark and Cooper ({\em On the adjacency spectra of hypertrees, Electron. J. Combin., 25 (2)(2018) $\#$P2.48}) to weighted $k$-trees.
As applications, two analogues of the above work for  the Laplacian  and the signless Laplacian tensors of $k$-trees are obtained.
\end{abstract}

\section{Introduction}

Since  Qi \cite{Qi} and Lim \cite{Lim} independently defined the concept of the tensor eigenvalue in 2005,  the problem of  tensor eigenvalues has received a lot of  attention, and has become an important subject of multilinear algebra and spectral hypergraph theory, see  \cite{Chen,Cooper,CooperLma,Ding,Ng,QiCommu}.
However, unlike the matrix eigenvalues,
for a symmetric tensor $\A$ of order $n$ and dimension $k\geq 3$,
it is difficult to calculate the eigenvalues of $\A$,
since the degree of the characteristic polynomial of $\A$ is $n(k-1)^{n-1}$ \cite{Qi}. The
computation of eigenvalues of higher order tensors is  NP-hard as mentioned by Hillar and Lim \cite{Hillar}.
In the present paper, we mainly study the eigenvalues of symmetric tensors whose underlying hypergraphs are $k$-trees, that is, the eigenvalues of adjacency tensors of weighted $k$-trees.

In 1972, Mowshowitz \cite{Mowshowitz} found that the coefficients of the characteristic polynomial of a tree can be expressed by the number of its matchings.
More precisely, the characteristic polynomial of a tree coincides with its matching polynomial. The reader may also see Corollary 4.2 of \cite{Godsil3} or Theorem 8.5.3 of \cite{Lovasz}.
This classical result is a fundamental identity in algebraic combinatorics, and plays an important role in spectral graph theory. For instance,
using this fact, Mowshowitz \cite{Mowshowitz} obtained that there exist infinitely many pairs of non-isomorphic cospectral trees
and Schwenk \cite{Schwenk} proved that almost all trees are not determined by their spectrum.

Inspired by Mowshowitz's classical work, recently, Zhang, Kang, Shan and Bai \cite{Zhang} proved the following:
For a $k$-tree $T$ with adjacency tensor $\A(T)$,
a nonzero number $\lambda$ is an eigenvalue of $\A(T)$ with the corresponding eigenvector having all elements nonzero
if and only if $\lambda$ is a root of the polynomial
 $$\sum_{i\geq 0}(-1)^ip(T,i) x^{(m(T)-i)k}=:\varphi(T,x),$$
where $m(T)$ is the matching number of $T$ and $p(T,i)$ denotes the number of $i$-matchings in $T$ with the convention that $p(T,0)=1$.
In particular, the spectral radius of $\A(T)$ is equal to the largest root of $\varphi(T, x)$.

Subsequently, Clark and Cooper \cite{ClarkCooper} discussed how to obtain all of the eigenvalues of a $k$-tree and gave the following Theorem.


\begin{theorem}[\cite{ClarkCooper}] \label{ClarkCoopertheorem11}
Let $T$ be a $k$-tree with $k\geq3$.
Then a nonzero number $\lambda$ is an eigenvalue of $\A(T)$ if and only if there exists an induced subtree $T'$ of $T$ such that $\lambda$ is a root $\varphi(T',x)$.
\end{theorem}

In this paper, we  extend Theorem \ref{ClarkCoopertheorem11} to weighted $k$-trees.
To begin with, let us introduce the definition of the weighted matching polynomials of weighted $k$-graphs.
Let $H$ be a $k$-graph and let $M$ be a subset of $E(H)$.
Denote by  $V(M)$ the set of vertices of $H$ each of which is an endpoint of one of the edges in $M$.
If no two distinct edges in $M$ share a common vertex, then $M$ is called a  {\it   matching} of $H$.
The set of matchings (including the empty set) of $H$ is denoted by $\mathcal{M}(H)$.
Let $\w: V(H)\cup E(H)\rightarrow \C$ be a weighting function on $H$.  We say $\w$ is {\it nonnegative} if $\w(v)\geq 0$ for any $v\in V(H)$ and $\w(e)> 0$ for any $e\in E(H)$.
The {\em weighted matching polynomial} of $(H,\w)$ is defined to be
 $$
\mu(H,\w,x)= \sum_{M \in \mathcal{M}(H)}(-1)^{|M|}\prod_{e \in E(M)}\mathbf{w}(e)^k \prod_{v \in V(H)\backslash V(M)}(x-\w(v)).
$$
Obverse that, if we choose the weighting function $\w$ such that $\w(v)=0$ for all $v\in V(H)$ and $\w(e)=1$ for all $e\in E(H)$, then $\mu(H,\w,x)$ is exactly the matching polynomial $\mu(H,x)$ of $H$ defined in \cite{Suejc}.
For a $k$-tree $T$ of order $n$, one may check that
$$\mu(T,x)=x^{n-km(T)}\varphi(T,x).$$
The matching polynomial of hypergraph is a natural extension of the matching polynomial of graph which is introduced by Heilmann and Lieb \cite{Heilmann}.
For more results in matching polynomial theory, we refer the reader to \cite{Godsil3,Godsil,Heilmann,Lovasz}.

We are now ready to show the main result of this paper, which is an extension of Theorem \ref{ClarkCoopertheorem11} and the work of Zhang, Kang, Shan and Bai \cite{Zhang}, and gives a complete characterization of the eigenvalues of symmetric tensor whose underlying hypergraph are a $k$-tree by means of the weighted matching polynomial.

\begin{theorem}\label{maintheorem}
Let $T$ be a $k$-tree equipped with a weighting function $\w: V(H)\cup E(H)\rightarrow \C$, where $k\geq3$.
Then, for every $v\in V(T)$, $\w(v)$ is an eigenvalue of the adjacency tensor $\A(T,\w)$ of $(T,\w)$.
Furthermore, if $\lambda\neq \w(v)$ for any $v\in V(T)$, then $\lambda$ is an eigenvalue of $\A(T,\w)$ if and only if
there exists a subtree $T'$ of $T$ such that $\lambda$ is a root of  $\mu(T',\w,x)$.
In particular, if $\w$ is nonnegative, then the spectral radius of $\A(T,\w)$ is equal to the largest root of $\mu(T,\w,x)$.
\end{theorem}


\begin{remark}
The condition `$k\geq3$' in Theorem \ref{maintheorem} is necessary, since the characteristic polynomial of $\A(T,\w)$ coincides with $\mu(T,\w,x)$ when $k=2$,
and it can be proved  by the classical recursive approach. See Theorem 8.5.3 of \cite{Lovasz} for the non-weighted version.
\end{remark}

Let $\L(T)$ (resp. $\Q(T)$) be the Laplacian tensor (resp. signless Laplacian tensor) of a $k$-tree $T$.
Clearly, if we choose the weight $\w$ such that $\w(v)=d_T(v)$ for all $v\in V(T)$ and $\w(e)=-1$ (resp. $1$) for all $e\in E(T)$,
then we immediately obtain an analogue of Theorem \ref{ClarkCoopertheorem11} for the Laplacian tensor (resp. signless Laplacian tensor) of $T$.

\begin{corollary}\label{corollarylq}
Let $T$ be a $k$-tree with $k\geq3$.
Then, for any $v\in V(T)$, the degree $d_T(v)$ of $v$ in $T$ is an eigenvalue of $\L(T)$ and $\Q(T)$, respectively.
Furthermore, if $\lambda\neq d_T(v)$ for any $v\in V(T)$, 
then
\begin{enumerate}

\item[{\rm (1)}]  $\lambda$ is an eigenvalue of $\L(T)$ if and only if there exists a subtree $T'$ of $T$ such that $\lambda$ is a root of the polynomial
$$\sum_{M \in \mathcal{M}(T')} (-1)^{(k+1)|M|} \prod_{v \in V(T')\backslash V(M)}(x-d_T(v)).$$

\item[{\rm (2)}]  $\lambda$ is an eigenvalue of $\Q(T)$ if and only if there exists a subtree $T'$ of $T$ such that $\lambda$ is a root of the polynomial
$$
\sum_{M \in \mathcal{M}(T')}(-1)^{|M|} \prod_{v \in V(T')\backslash V(M)}(x-d_T(v)).
$$
In particular, the spectral radius of $\Q(T)$ is equal to the largest root of the above polynomial.
\end{enumerate}
\end{corollary}
In this sense,  it should be appropriate that we say two polynomials displayed in Corollary \ref{corollarylq} are the {\it Laplacian matching polynomial} and  the {\it signless Laplacian matching polynomial} of $T'$ with respect to $T$, respectively.

The paper is organized as follows.
In Section 2, we give some basic definitions and results of tensors and the spectra of weighted uniform hypergraphs.
Section 3 is devoted to extend $\alpha$-normal method proposed by Lu and Man \cite{Lu2016} to study the eigenvalues of weighted uniform hypergraphs.
We prove Theorem \ref{maintheorem} in Section 4
and conclude this paper in the last section.

\section{Preliminaries}

\subsection{Eigenvalues of tensors}
A {\it tensor} (also called \emph{hypermatrix}) $\A=(a_{i_{1} i_2 \ldots i_{k}})$ of order $k$ and dimension $n$ refers to a multi-dimensional array with entries $a_{i_{1}i_2\ldots i_{k}}\in \C$ for all $i_{j}\in [n]:=\{1,2,\ldots,n\}$ and $j\in [k]$.
A tensor $\A=(a_{i_{1} i_2 \ldots i_{k}}) $ is called \emph{symmetric} if its entries are invariant under any permutation of their indices, that is,
$a_{i_{1}i_2\ldots i_{k}}=a_{i_{\sigma(1)}i_{\sigma(2)}\ldots i_{\sigma(k)}}$ for any permutations $\sigma$ on $[k]$. The associated digraph $D(\A)$ with $\A$ is a digraph with vertex set $[n]$, which has arcs $(i_1,i_2), \ldots, (i_1,i_k)$ for each nonzero entry $a_{i_{1} i_2 \ldots i_{k}}$
of $\A$. The tensor $\A$ is called {\it weakly irreducible} if $D(\A)$ is strongly connected \cite{Friedland}.

For a vector $\x=(\x_1, \x_2,\ldots,\x_n)^{\top}\in \mathbb{C}^{n}$,
let $\A \x^{k-1}$ be the vector in $\mathbb{C}^n$ whose $i$th component is defined by
\begin{align*}
(\A \x^{k-1})_i & =\sum_{i_{2},\ldots,i_{k}\in [n]}a_{ii_{2}\cdots i_{k}}\x_{i_{2}}\ldots \x_{i_k},
\end{align*} for any $i \in [n]$.
In 2005, Qi \cite{Qi} and Lim \cite{Lim} independently introduced the eigenvalues of tensors.
For a scalar $\lambda \in \mathbb{C}$, if the polynomial system 
$$\A \x^{k-1}=\lambda \x^{[k-1]},$$
where $\x^{[k-1]}:=(\x_1^{k-1}, \x_2^{k-1},\ldots,\x_n^{k-1})^{\top}$,
has a solution $\x \in \mathbb{C}^{n}\backslash \{{\mathbf{0}}\}$,
then $\lambda $ is called an \emph{eigenvalue} of $\A$ and $\x$ is called an \emph{eigenvector} of $\A$ associated with $\lambda$.
Moreover, the pair $(\lambda,\x)$ is called an \emph{eigenpair} of $\A$.
The \emph{determinant} of $\A$, denoted by $\det \A$, is defined as the resultant of the polynomials $\A  \x^{m-1}$,
and the \emph{characteristic polynomial} $\phi_\A(x)$ of $\A$ is defined as $\det(x\I-\A)$.
We know from \cite{Qi} that $\la$ is an eigenvalue of $\A$ if and only if it is a root of $\phi_\A(x)$.
The {\it spectral radius $\rho(\A)$} of $\A$ is the largest modulus of the eigenvalues of $\A$.
The Perron-Frobenius theorem was generalized from nonnegative matrices to nonnegative tensors. Here we list the parts of the theorem that we need.

\begin{lemma}[\cite{Friedland},\cite{YY}]\label{perronthm}
If $\A$ is a nonnegative weakly irreducible tensor of order $k$ and dimension $n$, then $\rho(\A)$ is an eigenvalue of $\A$ with the unique positive eigenvector, up to a positive scalar.
\end{lemma}

\subsection{Weighted uniform hypergraphs}

A {\it $k$-uniform hypergraph} (or simply  {\it $k$-graph}) $H=(V(H), E(H))$ consists of a finite vertex set $V(H)$
 and an edge set $E(H)$,
where each edge $e\in E(H)$ is a $k$-element subset of $V(H)$.
For $v\in V(H)$, denote by $E_H(v)$ the set of edges of $H$ containing $v$. The
cardinality of $E_H(v)$ is called the {\it degree} of $v$ in $H$, denoted by $d_H(v)$.

A {\it walk} in a $k$-graph $H$ is a sequence of alternate vertices and edges: $v_0 e_1 v_1 \ldots e_\ell v_\ell$,
where ${v_i, v_{i+1}}\in e_i$ for $i = 0, 1, \ldots, \ell-1$.
The walk $P$ is called a {\it path} if no vertices and edges appeared in $P$ are repeated, and is
called a {\it cycle} if no vertices and  edges appeared in $P$ are repeated, expect $v_0=v_{\ell}$.
A $k$-graph $H$ is {\it connected} if every two vertices of $H$ are connected by a walk
and is called a {\it $k$-uniform hypertree} (or simply  {\it $k$-tree}) if $H$ is in addition acyclic. 

If a $k$-graph $H$ equipped with a weighting function $\w: V(H)\cup E(H)\rightarrow \C$, then $(H,\w)$ is called {\it weighted $k$-uniform hypergraph} (or simply  {\it weighted $k$-graph}).
Let  $V(H)=\{v_1,v_2,\ldots,v_n\}$.
The {\it adjacency tensor} of $(H,\w)$ is defined as $\A(H,\w)=(a_{i_{1}i_{2}\cdots i_{k}})$,
a tensor of order $k$ dimensional $n$, where
\[a_{i_{1}i_{2}\cdots i_{k}}=\left\{
 \begin{array}{ll}
  \w(v) , &  \mbox{if~}  i_{1}=i_{2}=\cdots =i_{k}=v\in V(H);\\
\frac{\w(e)}{(k-1)!}, &  \mbox{if~} e=\{v_{i_{1}},v_{i_{2}},\cdots,v_{i_{k}}\} \in E(H);\\
  0, & \mbox{otherwise}.
  \end{array}\right.
\]
It is not hard to check that $\A(H,\w)$ is weak irreducible when $H$ is connected and $\w$ is nonnegative \cite{Friedland, YY}.

Clearly, if we choose the weight $\w$ such that $\w(v)=0$ for all $v\in V(H)$ and $\w(e)=1$ for all $e\in E(H)$, then $\A(H,\w)$ is exactly the adjacency tensor $\A(H)$ of $H$ defined in \cite{Cooper}.
Also, if we choose the weight $\w$ such that $\w(v)=d_H(v)$ for all $v\in V(H)$ and $\w(e)=-1$ (resp. $1$) for all $e\in E(H)$,
then $\A(H,\w)$ is exactly the Laplacian tensor $\L(T)$ (resp. signless Laplacian tensor $\Q(T)$) of $H$ defined in  \cite{QiCommu}.

Denote by $\x^U:=\Pi_{v \in U} \x_v$ for a subset $U$ of $V(H)$ and a vector $\x=(\x_1, \ldots,\x_n)^{\top}\in \C^{n}$.
Writing $\A= \A(H,\w)$,  we have
\begin{equation}\label{eigenvectorequation21}
(\A\x^{k-1})_v=\w(v)\x_v^{k-1}+\sum_{e\in E_H(v)}\w(e) \x^{e\backslash \{v\}}
\end{equation}
for each $v \in V(H)$.
So the  equation $\A \x^{k-1}=\lambda \x^{[k-1]}$ can be written as
\begin{equation}\label{eigenvectorequation}
\big(\lambda-\w(v)\big) \x_v^{k-1} =\sum_{e\in E_H(v)} \w(e)\x^{e\backslash \{v\}}
\end{equation}
for all $ v \in V(H)$.

\subsection{The eigenvalues of the weighted uniform hypertrees}

Let $\Gamma=(H,\w)$ be a weighted $k$-graph.
For a subset $U$ of $V(H)$, denote by $H[U]$ the subgraph of $H$ induced by $U$, that is,  a graph with $V(H[U])=U$ and $E(H[U])=\{e\in E(H):e \subset U\}$.
We shall use $\Gamma[U]:=(H[U],\w [U])$ to denote the weighted subgraph of $\Gamma$ induced by $U$,
where $\w [U]:=\w|_{U\cup E(H[U])}$. 
If the context is clear, we simply write $(H',\w)$ instead of $(H',\w[V(H')])$ when $H'$ is an induced subgraph of $H$.
An edge $e$ is called a {\it pendent edge} if it contains exactly $k-1$ vertices of degree one. We
use $H \setminus e$ to denote the $k$-graph obtained from $H$ by deleting  $e$ along with resultant isolated vertices.

\begin{lemma} \label{corevertex}
Let $(H,\w)$ be a weighted $k$-graph with $k\geq 3$, and let $e$ be an edge of $H$ containing at least two degree-one vertices.
If $\lambda$ is an eigenvalue of $\A(H\setminus e,\w)$, then $\lambda$ is an eigenvalue of $\A(H,\w)$.
\end{lemma}

\begin{proof}
Let $C(e)=\{v_1, \ldots, v_s\}$ be the set of all degree-one vertices in $e$, where $s\geq 2$.
Suppose that $(\lambda,\y)$ is an eigenpair of $\A(H\setminus e,\w)$.
Define a vector $\x\in \C^{n}$ by
\[ \x_v=\left\{
 \begin{array}{ll}
  \y_v , &  \mbox{if~} v\in V(H\setminus e);\\
  0, & \mbox{if~} v\in C(e).
  \end{array}\right.
\]
Write $\A= \A(H,\w)$.
For any $v\in V(H\setminus e)$, by  \eqref{eigenvectorequation}, one may check that
\begin{align*}
(\A \x^{k-1})_v &=\w(v)\x_v^{k-1}+\sum_{f\in E_H(v)}\w(f) \x^{f\backslash \{v\}}\\
& = \w(v)\y_v^{k-1}+\sum_{f\in E_{H\setminus e}(v)}\w(f)\y^{f\backslash \{v\}}\\
& = \lambda \y_v^{k-1} = \lambda \x_v^{k-1}.
\end{align*}
Also, for any $u\in C(e)$,
$$(\A\x^{k-1})_u=\w(u)\x_u^{k-1}+ \w(e) \x^{e\backslash \{u\}}=0= \lambda \x_u^{k-1} .$$
Thus, $(\lambda,\x)$ is an eigenpair of $\A$, as desired.
\end{proof}

\begin{lemma}[\cite{ClarkCooper}] \label{corevertexCC}
Let $T$ be a $k$-tree.
If $T'$ is a subtree of $T$, then there exists a sequence of edges $(e_1, \ldots, e_s)$ such that, for any $i$ $(1 \leq i \leq s)$, $e_i$ is a pendant edge of $T_{i-1}$,
 where $T_0= T$, $T_{i} = T_{i-1}\setminus e_i $, and $T_s = T'$.
\end{lemma}

Given a vector $\x\in \C^{n}$, the {\it support} of $\x$, denoted by $\s(\x)$,  is the set of all indices of nonzero coordinates of $\x$.
We are now ready to present the main result of this section.

\begin{theorem}\label{maintheorem321}
Let $(T,\w)$ be a weighted $k$-tree with $k\geq 3$.
Then, $\lambda$ is an eigenvalue of $\A(T,\w)$ if and only if there is a subtree $T'$ of $T$ such that $\lambda$ is an eigenvalue of $\A(T',\w)$ with the corresponding eigenvector having all elements nonzero.
\end{theorem}
\begin{proof}
The sufficiency follows from Lemma \ref{corevertex} and Lemma \ref{corevertexCC}.
For the necessity, let $(\lambda,\x)$ be an eigenpair of $\A(T,\w)$. Let $T'$ be a component of $T[\s(\x)]$. Also, let $\hat{\x}$ denote the nonzero projection (by restriction) of $\x$ onto  $\C^{|\s(\x)|}$.
 Write $\Gamma'=(T', \w[V(T')])$.
Then, for any $v\in V(T')$, by \eqref{eigenvectorequation}, one may check that
\begin{align*}
(\A(\Gamma') \hat{\x}^{k-1})_v
& = \w(v){\hat{\x}}_v^{k-1}+\sum_{e\in E_{T'}(v)}\w(e) {\hat{\x}}^{e\backslash \{v\}}\\
& = \w(v)\x_v^{k-1}+\sum_{e\in E_T(v)}\w(e) \x^{e\backslash \{v\}}\\
& = \lambda\x_v^{k-1} = \lambda{\hat{\x}}_v^{k-1},
\end{align*}
which implies that $(\lambda,\hat{\x})$ is an eigenvector of $\A(\Gamma')$.
The necessity follows from that each entry of $\hat{\x}$ is nonzero and $T'$ is a subtree of $T$.
\end{proof}

\section{ The $\alpha$-normal method for weighted uniform hypergraphs}

 In \cite{Lu2016}, Lu and Man introduced the $\alpha$-normal method to investigate the spectral radii of $k$-graphs.
The method is a powerful tool to deal with the eigenvalues of tensor \cite{Bai,Zhang} and the
$p$-spectral radii of $k$-graphs \cite{Liu}.
In this section, we generalize their method to weighted $k$-graphs by making some appropriate adjustments.

Let  $\Gamma=(H,\w)$ be a weighted $k$-graph. In what follows, we always assume that  $w(e) \neq 0$ for each $e\in E(H)$ unless we mention it.
A {\it weighted vertex-edge incidence matrix} $B$ of $\Gamma$ is a complex matrix whose rows are indexed by $V(H)$ and columns are indexed by $E(H)$ such that for any $v\in V(H)$ and any $e\in E(H)$, the entry $B(v, e)\neq 0$ if and only if  $v\in e$.

\begin{definition}\label{normaldefinition}
Let $\Gamma=(H,\w)$ be a weighted $k$-graph, and  let $\{\lambda_e\}_{e\in E(H)}$ be an ordered sequence consisting of complex numbers such that $\lambda_e\neq \w(v)$ for any $v\in e, e \in E(H)$. $\Gamma$ is called {\it $\{\lambda_e\}$-normal with respect to $B$} if there exists a weighted incidence matrix $B$ of $H$ satisfying
\begin{enumerate}

\item[{\rm (C1)}] $\sum\limits_{e \in E_H(v)} B(v,e)=1$ for any $v\in  V(H)$.

\item[{\rm (C2)}]
 $\prod\limits_{v \in e} B(v,e)=\prod\limits_{v \in e}\frac{\w(e) }{ \lambda_e-\w(v) }$ for any $e \in E(H)$.

\end{enumerate}
If $\lambda_e=\lambda$ for any $e \in E(H)$, then $\Gamma$ is called {\it $\lambda$-normal}.
 Moreover, $\Gamma$ is called {\it consistently $\lambda$-normal with respect to $B$} if $\Gamma$ is $\lambda$-normal and
\begin{enumerate}

\item[{\rm (C3)}] For any cycle $v_0,e_1,v_1, \ldots, e_\ell, v_\ell(=v_0)$ of $H$,
             $$ \prod_{i=1}^\ell \frac{B(v_i,e_i) (\lambda-\w(v_{i})) }{B(v_{i-1},e_i)(\lambda-\w(v_{i-1}))} =1.$$
\end{enumerate}
\end{definition}

\begin{remark}\label{fact}
Assume that $\Gamma$ is $\{\lambda_e \}$-normal with respect to $B$. If $v\in V(H)$ with $d_H(v)=1$, then $B(v,e)=1$ for $e\in E_H(v)$.
\end{remark}

In what follows, we simply say that $\Gamma$ is $\{\lambda_e \}$-normal or $\lambda$-normal if  the weighted incidence matrix $B$ is clear from the context.

\begin{lemma}\label{mainlemma1}
Let $\Gamma=(H,\w)$ be a connnect weighted $k$-graph of order $n$, and let $\x\in \C^{n}$.
If $\lambda\neq \w(v)$ for any $v\in V(H)$, then $(\lambda,\x)$ is an eigenpair of $\A(\Gamma)$ with $\s(\x)= V(H)$ if and only if $\Gamma$ is consistently $\lambda$-normal.
\end{lemma}

\begin{proof}
Assume that $\lambda$ is an eigenvalue of $\A(\Gamma)$ with an eigenvector $\x=(\x_1, \ldots,\x_n)^{\top}$ satisfying $\s(\x)= V(H)$ and  $\lambda\neq \w(v)$ for any $v\in V(H)$.
Define a weighted incidence matrix $B$ of $\Gamma$ by
 \[
 B(v,e)=\left\{
 \begin{array}{ll}
\frac{\w(e) \x^e}{(\lambda-\w(v)) \x_v^k}, &  \mbox{if~} v \in e;\\
  0, & \mbox{otherwise}.
  \end{array}\right.
\]
By  \eqref{eigenvectorequation}, one may check that for any $ v\in V(H)$,
$$\sum\limits_{e\in E_H(v)} B(v,e)= \frac{\sum_{e\in E_H(v)}\w(e)\x^e}{(\lambda-\w(v)) \x_v^k}=1,$$
for each edge $e\in E(H)$,
$$\prod\limits_{ v \in e}B(v,e)
=\prod\limits_{ v \in e}\frac{\w(e) }{ \lambda -\w(v) },$$
and for any cycle $v_0 e_1 v_1 \cdots e_\ell v_\ell(=v_0)$ of $H$,
\begin{align*}
 &\prod_{i=1}^\ell \frac{B(v_i,e_i) (\lambda-\w(v_{i}))}{B(v_{i-1},e_i) (\lambda-\w(v_{i-1}))}\\
=&\prod_{i=1}^\ell \frac{  ((\lambda-\w(v_{i})) { \w(e_i) \x^{e_i}} /{\left(( \lambda-\w(v_ i))    \x_{v_i}^k\right)}  }
                    {  ((\lambda-\w(v_{i-1})) { \w(e_i) \x^{e_i}} / {((\lambda-\w(v_{i-1})) \x_{v_{i-1}}^k)}  }\\
 =& \prod_{i=1}^\ell \frac { \x_{v_{i-1}}^k}  { \x_{v_i}^k } =1.
\end{align*}
This proves the necessity.

We now prove the sufficiency.
Assume that $\Gamma$ is consistently $\lambda$-normal with respect to a weighted incidence matrix $B$.
Construct a vector $\y\in \C^{n}$ as follows:
Fix $v_0 \in V(H)$, and let $\y_{v_0}=1$. Set
$$\y_u=\left(\prod_{i=1}^\ell \frac{B(v_{i-1},e_i) (\lambda-\w(v_{i-1}) )}{B(v_i,e_i)
(\lambda-\w(v_{i}) ) }  \right)^{1\over k} $$
if there exists a path $v_0 e_1 v_1 \ldots e_\ell v_\ell (v_\ell=u)$ in $H$.
Notice that such path always exists since $H$ is connected, and moreover Condition (C3) (in Definition \ref{normaldefinition}) ensures that the value of $\y_u$ is independent of the choice of the path.
Observe that $\y_v\neq 0$ for each vertex $v$.
Thus, it remains to prove that $(\lambda,\y)$ is an eigenpair of $\A(\Gamma)$.
By the construction of $\y$, for any $e=\{v_{i_1},\ldots,v_{i_k}\}\in E(H)$, it is obvious that
\begin{align}\label{equationindependent}
\big(B(v_{i_1}, e)(\lambda-\w(v_{i_1})\big) ^{\frac{1}{k}}\y_{v_{i_1}}
=&\big(B(v_{i_2}, e)(\lambda-\w(v_{i_2})\big) ^{\frac{1}{k}}\y_{v_{i_2}}\\ \nonumber
 =& \cdots=\big(B(v_{i_k}, e)(\lambda-\w(v_{i_k})\big) ^{\frac{1}{k}}\y_{v_{i_k}}.
\end{align}
By \eqref{equationindependent} and Condition (C2), one may check that,  for a given vertex $v\in V(H)$ and for any $e\in E_H(v)$,
\begin{align}\label{construction}
\y^e
=&\prod_{u\in e}\left(   \left( \frac{B(v,e)(\lambda-\w(v) )}{B(u,e)(\lambda-\w(u))} \right)^{1\over{k}} \y_v  \right )\\ \nonumber
 =& \frac{B(v,e)(\lambda-\w(v)) \y_v^k }{\left(\prod_{u\in e}B(u,e)(\lambda-\w(u))\right)  ^{1\over{k}}  }\\\nonumber
  =&\frac{B(v,e)(\lambda -\w(v)) \y_v^k }{  \w(e)  }.
\end{align}
Using \eqref{construction} and Condition (C1), we can deduce that
$$
\sum_{ e\in E_H(v)}\w(e) \y^e = \sum_{ e\in E_H(v)} B(v,e)(\lambda -\w(v)) \y_v^k=(\lambda -\w(v)) \y_v^k,
$$
Thus, $(\lambda,\y)$ is an eigenpair of $\A(\Gamma)$ with $\s(\y)=V(H)$, as desired.
\end{proof}

Combining Theorem \ref{maintheorem321} and Lemma \ref{mainlemma1}, we immediately obtain the following result, which plays an important role in the proof of Theorem \ref{maintheorem}.

\begin{theorem}\label{corollary322}
Let $(T,\w)$ be a weighted $k$-tree with $k\geq 3$.
If $\lambda\neq \w(v)$ for any $v\in V(T)$,
then $\lambda$ is an eigenvalue of $\A(T,\w)$ if and only if there exists a subtree $T'$ of $T$ such that
$(T',\w)$ is consistently $\lambda$-normal.
\end{theorem}

\section{Proof of Theorem \ref{maintheorem}}

 We prove Theorem \ref{maintheorem} in this section.
To do it, we need to introduce a multivariate polynomial.
For a weighted $k$-tree $\Gamma=(T,\w)$,
let $\tilde{\mu}(\Gamma)$ denote the following multivariate polynomial with respect to indeterminates $\{x_e \}_{e\in E(T)}$:
\begin{equation}\label{defmatp44}
\tilde{\mu}(\Gamma)=\tilde{\mu}\left( \Gamma,\{x_e \}_{e\in E(T)}\right)
:= \sum_{M \in \mathcal{M}(T)} (-1)^{|M|}
\prod_{e \in E(M)} \left(\prod_{v \in e} \frac{\w(e)}{x_e-\w(v)}\right).
\end{equation}

\begin{lemma}\label{lemma4111}
Let $\Gamma=(T,\w)$ be a weighted $k$-tree, and let  $\{\lambda_e\}_{e\in E(T)}$ be an ordered sequence consisting of complex numbers and satisfying $\lambda_e\neq \w(v)$ for any $v\in e$.
Then $\Gamma$ is $\{\lambda_e\}$-normal if and only if $\{\lambda_e\}_{e\in E(T)}$ is a root of $\tilde{\mu}(\Gamma)$.
 \end{lemma}
\begin{proof}
We prove the assertion by induction on $|E(T)|$, say $m$.
If $m=1$, then $T$ contains only one edge, say $e$.
If $\Gamma$ is $\{\lambda_e\}$-normal with respect to $B$, by Definition \ref{normaldefinition} and Remark \ref{fact}, then we derive that $B(v,e)=1$ for any $v\in V(T)$ and
$$\prod\limits_{v \in e} B(v,e)=\prod\limits_{v \in e}\frac{\w(e) }{ \lambda_e-\w(v) },$$
which implies
 $$\prod\limits_{v \in e}\frac{\w(e) }{ \lambda_e-\w(v) }=1.$$
Thus, $\{\lambda_e\}_{e\in E(T)}$ is a root of $\tilde{\mu}(\Gamma)$.
Conversely, assume that $\{\lambda_e\}_{e\in E(T)}$ is a root of $\tilde{\mu}(\Gamma)$, equivalently,
 $$\prod\limits_{v \in e}\frac{\w(e) }{ \lambda_e-\w(v) }=1.$$
Let $B$ be the weighted incidence matrix of $\Gamma$ such that $B(v,e)=1$ for any $v\in e$.
It is obvious that $\Gamma$ is $\lambda_e$-normal with respect to $B$, as desired.

Now, suppose that $m \geq 2$ and the assertion holds for natural numbers less than $m$.
Consider a weighted $k$-tree $\Gamma$ of size $m$.
Suppose that $\Gamma$ is $\{\lambda_e\}$-normal.
Since $T$ is a $k$-tree, there exists a vertex $u$ satisying $d_{T}(u) \geq 2$ and  incident with at most one non-pendent edge in $T$.
Denote by $P(u)$  the set consisting of  $d_T(u)-1$ chosen pendent edges incident to $u$ and by $f$ the remaining edge.
Deleting all edges in $P(u)$ and the resultant isolated vertices, we obtain a
 subtree of $T$, denoted by $\widehat{T}$.

Take an edge $g\in P(u)$.
By Remark \ref{fact}, $B(v,g)=1$ for any $v\in g \setminus\{u\}$.
This implies that
$$B(u,g)=\prod\limits_{v \in g} B(v,g)=\prod\limits_{v \in g}\frac{\w(g) }{ \lambda_g-\w(v) }$$
by Condition (C2).
Using Condition (C1), we have
$$B(u,f)= 1-\sum\limits_{g \in P(u)} B(u,g)=1-\sum\limits_{g \in P(u)} \prod\limits_{ v \in g}\frac{\w(g) }{ \lambda_g-\w(v) }.$$
Define a weighted incidence matrix $\widehat{B}$ of $\widehat{T}$ by
\[\widehat{B}(v,e)=\left\{
 \begin{array}{ll}
  B (v,e), &  \mbox{if~}  v\in V(\widehat{T})\setminus\{u\} \mbox{\ and~}  e\in E(\widehat{T}) ;\\
1, &  \mbox{if~} v=u \mbox{\ and~}  e=f.
  \end{array}\right.
\]
Let $\widehat{\lambda}_f$ be a solution of the equation (with respect to  indeterminates $x$)
\begin{equation}\label{proof400}
\frac{\prod\limits_{ v \in f}\frac{\w(f) }{ \lambda_f-\w(v) }  }
{1-\sum\limits_{g \in P(u)} \prod\limits_{ v \in g}\frac{\w(g) }{ \lambda_g-\w(v) }  }
= \prod\limits_{ w \in f} \frac{\w(f) }{ x-\w(v) },
\end{equation}
and set $\widehat{\lambda}_e=\lambda _e$ for any $e\in  E(\widehat{T})\setminus\{f\}$.
Now, one may check that $\widehat{\Gamma}=(\widehat{T},\w[V(\widehat{T})])$ is $\{\widehat{\lambda}_e\}_{e\in E(\widehat{T})}$-normal  with respect to $\widehat{B}$.
By the induction hypothesis, we obtain that $\{\widehat{\lambda}_e\}_{e\in E(\widehat{T})}$ is a root of the polynomial
\begin{equation}\label{proof411}
\sum_{M \in \mathcal{M}(\widehat{T})} (-1)^{|M|}
\prod_{e \in E(M)} \left(\prod_{v \in e} \frac{\w(e)}{x_e-\w(v)}\right).
\end{equation}
Denote by $\mathcal{R}(\widehat{T})$ the set of all matchings of $\hat{T}$ that contains the edge $f$.
Write $\mathcal{S}(\widehat{T}) = \mathcal{M}(\widehat{T}) \setminus \mathcal{R}(\widehat{T})$ and $\chi_M=\prod_{e \in E(M)} \left(\prod_{v \in e} \frac{\w(e)}{\lambda_e-\w(v)}  \right)$ for any matching $M$ of $T$.
From \eqref{proof400} and \eqref{proof411}, one may check that
$$
 \sum_{M \in \mathcal{R}(\widehat{T})}(-1)^{|M|}
\left(
\frac{ 1  } {1-\sum\limits_{ g \in P(u)} \prod\limits_{ v \in g}\frac{\w(g) }{ \lambda_g-\w(v) }  }
\chi_M
\right)
  +   \sum_{M \in \mathcal{S} (\widehat{T})} (-1)^{|M|}\chi_M=0.
$$
Equivalently,
$$
 \sum_{M \in \mathcal{R}(\hat{T})}(-1)^{|M|}\chi_M+ \sum_{M \in \mathcal{S} (\hat{T})} (-1)^{|M|}\chi_M
 -\left(\sum\limits_{ g \in P(u)} \chi_g \right)
 \left(\sum_{M \in \mathcal{S} (\hat{T})} (-1)^{|M|}\chi_M\right)
 =0,
$$
which implies that
\begin{equation}\label{proof422}
 \sum_{M \in \mathcal{R} } (-1)^{|M|}\chi_M
+ \sum_{M \in \mathcal{S} } (-1)^{|M|}\chi_M
+ \sum_{M \in \mathcal{M}(T) \setminus  (\mathcal{R}\cup \mathcal{S})} (-1)^{|M|}\chi_M =0,
\end{equation}
where $\mathcal{R}$ denotes the set of all matchings of $T$ containing the edge $f$ and $\mathcal{S}$ denotes the set of all matchings of $T$ containing none of $E_H(u)$.
By \eqref{proof422}, we deduce that $\{\lambda_e\}_{e\in E(T)}$ is a root of $\tilde{\mu}(\Gamma)$.

The sufficiency can be proved by conversing the above steps, and so the result follows.
\end{proof}

We are now ready to prove Theorem \ref{maintheorem}.

\begin{proof}[\proofname{ of \bf Theorem \ref{maintheorem}.}]
Write $\A= \A(\Gamma)$.
For a vertex $v\in V(H)$, let $\v\in \C^{n}$ be the vector whose $v$-th entry is $1$ and other entries are $0$.
Using \eqref{eigenvectorequation21} and $k\geq 3$, we find that
$$(\A\v^{k-1})_v=\w(v)\v_v^{k-1}+\sum_{e\in E_H(v)}\w(e) \v^{e\backslash \{v\}}=\w(v)\v_v^{k-1}$$
and for any $u\neq v$,
$$(\A\v^{k-1})_u=\w(u)\v_u^{k-1}+\sum_{e\in E_H(u)}\w(e) \v^{e\backslash \{u\}}=0=\w(v)\v_u^{k-1}.$$
Thus, $(\w(v),\v)$ is an eigenpair of $\A$.
The first statement follows the arbitrariness of $v$.

If $\lambda\neq \w(v)$ for any $v\in V(T)$,  by Theorem \ref{corollary322}, then $\lambda$ is an eigenvalue of $\A(\Gamma)$ if and only if there is a subtree $T'$ of $T$ such that $\Gamma[V(T')]$ is consistently $\lambda$-normal.
By Lemma \ref{lemma4111}, it is equivalent to that $\lambda$ is a root of the polynomial
$$\sum_{M \in \mathcal{M}(T')} (-1)^{|M|}
\prod_{e \in E(M)} \left(\prod_{v \in e} \frac{\w(e)}{x -\w(v)}\right).
$$
Since $\lambda\neq \w(v)$ for any $v\in V(T)$,
$\lambda$ is a root of the polynomial
\begin{align*}
&\left(\prod_{v \in V(T')} (x -\w(v))  \right)  \sum_{M \in \mathcal{M}(T')} (-1)^{|M|}
\prod_{e \in E(M)} \left(\prod_{v \in e} \frac{\w(e)}{x -\w(v)}\right),
\end{align*}
which is equal to
$$\sum_{M \in \mathcal{M}(T')}(-1)^{|M|}\prod_{e \in E(M)}\mathbf{w}(e)^k \prod_{v \in V(T')\backslash V(M)}(x-\w(v))
 =\mu(T',\w,x).$$
This shows the second statement of Theorem \ref{maintheorem} under the hypothesis in which $w(e) \neq 0$ for each $e\in E(T)$. If not, deleting all edges whose weights are  equal to $0$, we obtain a subforest of $T$, denoted by $F$. It is easy to see that $\lambda$ is an eigenvalue of some component of $F$. Applying the above discussion to this component, the second statement follows.

Finally, if $\w$ is nonnegative, then $\A(T,\w)$ is weak irreducible
and so that the spectral radius of $\A(T,\w)$ is an eigenvalue of $\A(T,\w)$ with a positive eigenvector $\x$ by Lemma \ref{perronthm}.
Since $\x$ is positive, by Lemma \ref{mainlemma1} and Lemma \ref{lemma4111}, the spectral radius of $\A(T,\w)$ is a root of $\mu(T,\w,x)$.
Note that the spectral radius of $\A(T,\w)$ is the maximum module over all eigenvalues of $\A(T,\w)$,
so it must equal the largest root $\mu(T,\w,x)$ by the second statement.
The proof is completed.
\end{proof}

\section{Concluding Remarks}
In this paper, we mainly give a complete characterization of the eigenvalues of the adjacency tensors of weighted $k$-trees by means of the modified $\alpha$-normal method and the weighted matching polynomial.
As we mentioned before, for a $k$-graph $H$ of order $n$ with $k\geq 3$, the degree of the characteristic polynomial of $H$ is $n(k-1)^{n-1}$ \cite{Cooper}, but the degree of the matching polynomial of $H$ is always $n$.
In fact, we were unable to compute characteristic polynomials for some $3$-graphs on only $9$ vertices  \cite{Cooper}.
However, the degree of the matching polynomial of $T$ is always $n$,
so Theorem \ref{ClarkCoopertheorem11} and Theorem \ref{maintheorem} suggest that
when we calculate the eigenvalues of the adjacency tensor of a weighted uniform hypertree, it suffices to find the zeros of the weighted matching polynomials of its weighted subtrees.
It seems that the later problem is easier for some special weighted uniform hypertrees.

As we know, the characteristic polynomial of a $2$-tree coincides with its matching polynomial \cite{Lovasz}.
Thus, for a $2$-tree $T$, $\lambda$ is a root of $\mu(T,x)$ with multiplicity $m$ if and only if  $\lambda$ is a root of $\phi_{\A(T)}(x)$ with multiplicity $m$.
However, for a $k$-tree $T$ with $k\geq 3$, Theorem \ref{ClarkCoopertheorem11} and Theorem \ref{maintheorem} tell us that
a nonzero number $\lambda$ is an eigenvalue of $\A(T)$ if and only if there exists an induced subtree $T'$ of $T$ such that $\lambda$ is a root $\varphi(T',x)$.
This raises a natural question: Is it possible to determine the multiplicity of an eigenvalue of $\A(T)$ as a root of the characteristic polynomial $\phi_{\A(T)}(x)$? For further discussion, see  \cite{Hu,Fan,CooperFickes}.
Clark and Cooper \cite{ClarkCooper} conjectured the following:
If $T'$ is a subtree of a $k$-tree $T$ with $k\geq 3$, then $\varphi(T',x)$ divides $\phi_{\A(T)}(x)$, where
$$\varphi(T',x):=\sum_{r\geq 0}(-1)^rp(T',r) x^{(m(T')-r)k}.$$
We wonder that the conjecture is also true for a weighed $k$-tree $(T,\w)$, that is,
if $T'$ is a subtree of $T$, then $\mu(T',\w,x)$ divides the characteristic polynomial of $\A(T,\w)$.



\end{document}